\documentclass[11pt]{article}
\usepackage{amsmath, amsthm, amssymb, amsfonts, latexsym,amscd}
\usepackage{authblk}
\usepackage{tikz}
\usetikzlibrary{graphs,graphs.standard} 
\usepackage{mathtools}
\usetikzlibrary{shadings,patterns}
\newtheorem{theorem}{Theorem}[section]
\newtheorem{corollary}[theorem]{Corollary}
\newtheorem{lemma}[theorem]{Lemma}
\newtheorem{proposition}[theorem]{Proposition}

\newtheorem{remark}[theorem]{Remark}
\newtheorem{example}[theorem]{Example}

\begin{document}

\title{\bf On the distances between the core center and other central parts of a tree}

\author[1]{Akash De \thanks{Corresponding Author: akash.de.2023@niser.ac.in} \thanks{Supported by the Council of Scientific and Industrial Research(CSIR) Grant No. 09/1002(17117)/2023-EMR-I} }
\author[2]{Kamal Lochan Patra \thanks{klpatra@niser.ac.in} }
\affil[1,2] {School of Mathematical Sciences,
National Institute of Science Education and Research Bhubaneswar,
P.O.- Jatni, District- Khurda, Odisha - 752050, India 
}
\affil[1,2] { Homi Bhabha National Institute(HBNI), 
Training School Complex, Anushakti Nagar, Mumbai - 400094, India 
}
\date{}
\maketitle
\begin{abstract}
Let $T$ be a tree. For a vertex $v\in V(T)$, the  eccentric subtree number $\epsilon_T(v)$ is defined as $\epsilon_T(v)=\min\{f_T(v,u): u\in V(T)\}$ where $f_T(v,u)$ denotes the number of subtrees of $T$ containing both $v$ and $u$. A  core vertex of $T$ is a vertex with the maximum eccentric subtree number, and the set of all the core vertices of $T$ is called the  core center of $T$. The core center of $T$ consists of either a single vertex or two adjacent vertices. There are other central concepts in a tree, such as the center, centroid, subtree core, and characteristic center, and these may all be different.
 
By $d_T(C, \mathfrak{C})$, $d_T(C_d, \mathfrak{C})$ and $d_T(S_c, \mathfrak{C})$ we mean the distance between center and core center, distance between centroid and core center and distance between subtree core and core center in $T$, respectively. We show that for any tree $T$ on $n\geq 6$ vertices,
\begin{enumerate}
\item[(i)]$d_T(C,\mathfrak{C})\leq \lfloor \frac{n-g_0-4}{2} \rfloor$;
\item[(ii)]$d_T(C_d,\mathfrak{C})\leq \lfloor \frac{n-5}{2} \rfloor$;
\item[(iii)] $d_T(S_c,\mathfrak{C})\leq\left\{
\begin{array}{ll}
 1,  &\text{if $n=7$,}\\
 n-g_0-3, &\text{if $n\neq 7$;}\\
\end{array}
\right.$
\end{enumerate}
where $g_0\geq 2$ be the smallest positive integer such that $2^{g_0-1}+g_0\geq n-3$. Moreover, we show that these bounds are best possible by obtaining a tree which attains these bounds. We also obtain a tree which maximizes the distance between characteristic center and core center over all trees on $n\geq 6$ vertices. The asymptotic behaviour of all these distances are also studied.\\


\end{abstract}

\section{Introduction}

Let $T$ be a tree with vertex set $V(T)$ and edge set $E(T)$. A vertex of degree one in $T$ is called a pendant vertex. For $u,v\in V(T)$, the distance $d_T(u,v)$(or simply $d(u,v)$) between $u$ and $v$, is the number of edges in the $u$-$v$ path(path joining $u$ and $v$). For two subsets $X$ and $Y$ of $V(T)$, the distance $d(X,Y)$ between $X$ and $Y$ is defined as $d(X,Y)=\min\{d(x,y)|x\in X, y\in Y\}$. For $v\in V(T)$, $e(v)=\max\{d(v,u)|u\in V(T)\}$ is called the {\it eccentricity} of $v$ in $T$. The {\it radius} rad$(T)$ of $T$ is defined as rad$(T)=\min\{e(v)|v\in V(T)\}$  and the {\it diameter} diam$(T)$ of $T$ is defined as diam$(T)=\max\{e(v)|v\in V(T)\}$. It is clear that diam$(T)=\max\{d(u,v)|u,v\in V(T)\}.$ The {\it center} $C(T)$ of $T$ is defined as $C(T)=\{v\in V(T)|e(v)=rad(T)\}$.
 
For $v\in V(T)$, a {\it branch} at $v$ is a maximal subtree of $T$ containing $v$ as a pendant vertex. Note that the number of branches at $v$ is equal to the degree of $v$. The {\it weight} of $v$, denoted by $\omega_T(v)$(or simply $\omega(v)$) is the maximum number of edges contained in a branch at $v.$ The {\it centroid} $C_d(T)$ of $T$ is the set of vertices $v$ for which $\omega(v)$ attain its  minimum. The following result is due to Jordan \cite{Jor}.

\begin{theorem}[\cite{Bl}, Theorem 2.1, Theorem 2.3]\label{B and H}

Let $T$ be a tree. Then
\begin{enumerate}
\item[$(i)$] $C(T)$ consists of either a single vertex or two adjacent vertices;
\item[$(ii)$] $C_d(T)$ consists of either a single vertex or two adjacent vertices.
\end{enumerate}
\end{theorem}

Both center and centroid are considered as central parts of a tree and they may be different. The concept of central parts in trees were started by Jordan (\cite{Jor}) in 1869 with the definitions of center and centroid. Later many researchers contributed to this study and define several other central parts like median, security center, telephone center, accretion center, weight balance center, latency center, pairing center, processing center, n-th power center of gravity, distance balance center etc. For more on different central parts of trees, we refer to the survey paper by Reid \cite{Reid} and the reference therein. Most of these above mentioned central parts are defined for a tree $T$ and are same as either $C(T)$ or $C_d(T)$. In 2005, Sz\'ekely and Wang introduced a new central part of a tree in \cite{Sze}, which is different from both center and centroid.

For $v\in V(T)$, $f_T(v)$ is the number of subtree of $T$ containing $v$. The {\it subtree core} of $T$ is the set of vertices for which $f_T(v)$ attains its maximum. We denote the subtree core of $T$ by $S_c(T)$. Note that for $|V(T)|\geq 3$, $S_c(T)$ does not contain any pendant vertex (see \cite{Des}, Remark 1.5). The following result is due to  Sz\'ekely and Wang.

\begin{theorem}[\cite{Sze}, Theorem 9.1]\label{S and W}
The subtree core of a tree consists of either a single vertex or two adjacent vertices.
\end{theorem}

Recently in \cite{Zhang}, Zhang et al. introduced a new central part of a tree $T$ and called it as core vertices of $T$. In \cite{Pandey3}, the authors extended this to any connected graph and termed it as core center of a connected graph. Based on this, the core center of a  tree is defined as the following. For a vertex $v$ of  $T$, the {\it eccentric subtree number} $\epsilon_T(v)$(or simply $\epsilon(v)$) is defined as $\epsilon_T(v)=\min\{f_T(v,u): u\in V(T)\}$ where $f_T(v,u)$ denotes the number of subtrees of $T$ containing both $v$ and $u$. A {\it core vertex} of $T$ is a vertex with the maximum eccentric subtree number, and the set of all core vertices is called the {\it core center} of $T$. We denote the  core center of $T$ by $\mathfrak{C}(T)$. Like center, centroid and subtree core, it  is observed that for $|V(T)|\geq 3$,  $\mathfrak{C}(T)$ does not contain any pendant vertex (see Remark \ref{reccentric}).
 
\begin{theorem}[\cite{Zhang}, Theorem 3.4 and \cite{Pandey3}, Theorem 4.13]\label{zha}
The core centre of a tree consists of either one vertex or two adjacent vertices.
\end{theorem}

All the central parts discussed before are combinatorially defined. We now give an algebraic definition of a central part of a tree, which is different from  the  above combinatorially defined central parts. Let $G$ be a simple graph on $n$ vertices with $V(G)=\{v_1,v_2,\ldots,v_n\}$.  The adjacency matrix $A(G)=(a_{ij})$ of $G$ is the $n\times n$ matrix with $a_{ij}=1$ if $v_i$ and $v_j$ are adjacent and $0$ otherwise, for $1\leq i,j \leq n$. The Laplacian matrix $L(G)$ of $G$ is defined as $L(G)=D(G)-A(G),$ where $D(G)=(d_{ij})$  is the $n\times n$ diagonal matrix with $d_{ii}=deg(v_i)$ for $i=1,2,\ldots,n$. It is well known that $L(G)$ is a real symmetric positive semi-definite  matrix. The smallest eigenvalue of $L(G)$ is $0$ with all one vector as an eigenvector. The second smallest eigenvalue of $L(G)$  is positive if and only if $G$ is connected (see \cite{Fie1}). The second smallest eigenvalue of $L(G)$ is called the {\it algebraic connectivity} of $G$ and we  denote it by $\mu(G)$. An eigenvector corresponding to $\mu(G)$ is called a {\it Fiedler vector} of $G.$

Let $Y$ be a Fiedler vector of $G$. By $Y(v)$ we mean the co-ordinate of $Y$ corresponding to the vertex $v$ of $G$. A vertex $v$ is called a {\it characteristic vertex}  of $G$ with respect to (w.r.t.) $Y$ if it satisfies one of the following two conditions.
 \begin{enumerate}
 \item[$(i)$] $Y(v)=0$ and there exists a vertex $u$ adjacent to $v$ such that $Y(u)\neq 0$.
 \item[$(ii)$] there exists a vertex $u$ adjacent to $v$ such that $Y(v)Y(u)<0$.
 \end{enumerate}
 \noindent The set of all characteristic vertices of $G$ w.r.t. $Y$ is called the {\it characteristic set} of $G$ w.r.t. $Y$. We denote the characteristic set of $G$ w.r.t.  $Y$ by $\chi_c(G,Y)$. The {\it characteristic center} $\chi_c(G)$ of $G$ is defined as
$$\chi_c(G)=\{v\in V(G): v \in \chi_c(G,Y) \;\;\mbox{for some Fiedler vector $Y$}\}.$$

\begin{theorem}[\cite{Pandey3}, Proposition 5.11]\label{p and p}
The characteristic center of a tree consists of either a single vertex or two adjacent vertices.
\end{theorem}
Note that for a tree $T$, $\chi_c(T,Y)$ is fixed for any Fiedler vector $Y$ (see \cite{Fie2}, Theorem 3,14 and \cite{Rm}, Theorem 2). The center, centroid, subtree core, core center and characteristic center may all be the same in a tree. For example, in the path $P_n:12\cdots n$ and the star $K_{1,n-1}$, it can be easily verified  that 
\begin{equation*}\label{CMS_path}
C(P_n)=C_d(P_n)=S_c(P_n)=\mathfrak{C}(P_n)=\chi_c(P_n)=
\begin{cases}
\{{\frac{n}{2}},{\frac{n}{2}+1}\}  &\mbox{if n is even},\\ 
\{{\frac{n+1}{2}}\}   & \mbox{if n is odd}
\end{cases}
\end{equation*}
and 
\begin{equation*}\label{CMS_star}
C(K_{1,n-1})=C_d(K_{1,n-1})=S_c(K_{1,n-1})=\mathfrak{C}(K_{1,n-1})=\chi_c(K_{1,n-1})=\{v\}
\end{equation*}
\noindent where $v\in V(K_{1,n-1})$ is the vertex of degree $n-1$. In Example 5.8 of \cite{Pandey3}, the authors have provided a tree in which all these central parts are different. 

Among all trees on $n$ vertices, the problems of maximizing the pairwise distances between center, centroid and characteristic center are first studied in  \cite{Klp}. Later different authors have contributed to this study and considered more interesting problems related to these pairwise distances (see \cite{Des},\cite{Pandey2},\cite{Pandey4} and \cite{Hs}). The asymptotic nature of these distances are also explored by different researchers (see \cite{Na} and \cite{Pandey2}). In this article, we consider the problems of maximizing the distances between center and core center, centroid and core center, subtree core and core center, and characteristic center and core center over trees on $n$ vertices. 

\subsection{Organization of the paper}

In Section \ref{Prim}, we explore some properties of eccentric subtree number of a tree  $T$. We also introduce the notion of eccentric diameter of $T$ and discuss some of its features. We start Section \ref{pst} with an introduction to path-star trees. We then discuss the position and ordering of center, centroid, characteristic center, subtree core and core center in a path-star tree. In Section \ref{pcc}, we study the effect on the position of core center of a tree under some perturbations. Section \ref{dcp} is dedicated for the study of maximizing the distances between center and core center, centroid and core center, subtree core and core center, and characteristic center and core center over trees on $n$ vertices. We obtain strict upper bounds for these distances and also some trees which attain these bounds. Subsequently, we discuss the asymptotic behaviour of these distances.

\section{Preliminaries}\label{Prim}

Let $T$ be a tree. The eccentric diameter $\epsilon^d(T)$ of $T$ is defined as $\epsilon^d(T)=\min\{f_T(u,v): u, v\in V(T)\}$. The following results on eccentric subtree number and eccentric diameter are due to Zhang et al.

\begin{theorem}[\cite{Zhang}, Proposition 2.1, Theorem 2.3]\label{zha1}
Let $T$ be a tree. Then the following hold:
\begin{enumerate}
\item[(i)] For $v\in V(T),$ if $\epsilon(v)= f_T(v,u)$ then $u$ is a pendant vertex;
\item[(ii)] If $\epsilon^d(T)=f_T(u,v)$ then both $u$ and $v$ are pendant vertices;
\item[(iii)] If $\epsilon^d(T)=f_T(u,v)$ for some pendant vertices $u$ and $v$ then $\epsilon(w)=\min\{f_T(w,u), f_T(w,v)\}$ for any $w\in V(T)$.
\end{enumerate}
\end{theorem}
For $v_1,v_2,\ldots,v_k\in V(T)$, $f_T(v_1,v_2,\ldots,v_k)$ denotes the number of subtrees of T containing $v_1,v_2,\ldots,v_k$. By $f_T(v_1,\ldots,v_{k-1},\overline{v_k})$, we mean the number of subtrees of $T$ containing $v_1,\ldots,v_{k-1}$ but not $v_k$.

\begin{remark}\label{reccentric}
Like center, centroid and subtree core, the core centre of a tree on $n\geq 3$ vertices does not contain any pendent vertex.
\end{remark}
This can be seen as follows. Let $x$ be a pendant vertex of $T$ and let $xw \in E(T)$. Let $\epsilon^d(T)=f_T(v,u)$ for some pendant vertices $u$ and $v$. Here $w$ is not a pendant vertex as $n\geq 3$. By Theorem \ref{zha1}(iii), $\epsilon(x)=\min\{f_T(x,u),f_T(x,v)\}\;\; \mbox{and}\;\; \epsilon(w)=\min\{f_T(w,u), f_T(w,v)\}$.  Suppose $x$ is neither equal to $u$ nor equal to $v$. Then $f_T(x,u)=f_T(x,w,u)<f_T(w,u)$ and $f_T(x,v)=f_T(x,w,v)<f_T(w,v)$. So $\epsilon(x)=\min\{f_T(x,u), f_T(x,v)\}<\min\{f_T(w,u), f_T(w,v)\}=\epsilon(w)$. Otherwise, without loss of generality, take $x=u$. Then $\epsilon(x)=f_T(x,v)=f_T(x,w,v)<\min\{f_T(w,x), f_T(w,v)\}=\epsilon(w)$. So, $\epsilon(x)<\epsilon(w)$ and hence $x\notin \mathfrak{C}(T)$.

Let $u,w,v$ be three vertices of $T$ with $uw,wv \in E(T)$. Then $2e(w)\leq e(u)+e(v)$ where $e$ is the eccentricity function on $V(T)$ (\cite{Pandey3}, Theorem 3.5). A result analogous to this on the eccentric subtree number function on $V(T)$ is proved in \cite{Zhang}.
\begin{theorem}[\cite{Zhang}, Theorem 3.1]\label{zha2}
Suppose $u,v,w\in V(T)$ such that $uw,wv \in E(T)$. Then $2\epsilon(w)\geq \epsilon(u)+\epsilon(v)$ with a possible equality only if $deg(w)=2$. 
\end{theorem}
In \cite{Zhang}, the authors have used Theorem \ref{zha2} to prove that $\mathfrak{C}(T)$ contains either one vertex or two adjacent vertices. The following corollary is very useful.
\begin{corollary}\label{cor1}
Let $T$ be a tree and let $P:v_1v_2\cdots v_k$ be a $v_1$-$v_k$ path in $T$. If $\epsilon(v_1)>\epsilon(v_2)$  then $\epsilon(v_1)>\epsilon(v_2)>\cdots>\epsilon(v_k)$. 
\end{corollary}
\begin{proof} By Theorem \ref{zha2}, we have $2\epsilon(v_2)\geq\epsilon(v_1)+\epsilon(v_3)$ and this implies $\epsilon(v_2)-\epsilon(v_3)\geq \epsilon(v_1)-\epsilon(v_2)>0$. So, $\epsilon(v_2)>\epsilon(v_3)$ and continuing in this way we get  $\epsilon(v_1)>\epsilon(v_2)>\ldots>\epsilon(v_k)$.  
\end{proof}

\begin{lemma}\label{lem2}
Let $T$ be a tree with $uv\in E(T)$. Suppose $X$ and $Y$ are the components of ${T}-uv$ containing $u$ and $v$, respectively. If $\epsilon(u)\geq \epsilon(v)$ and $\epsilon(v)=f_T(v,w)$ then $w$ is in $X$.    
\end{lemma}
\begin{proof} Assume that  $w$ is in $Y$. Then ${\epsilon}(u)\leq f_T(u,w)=f_T(u,v,w)<f_T(v,w)={\epsilon}(v)$, which is a contradiction. So $w$ is in $X$.  
\end{proof}

Note that for a tree $T$, if $diam(T)=d(u,v)$ then the center $C(T)$ lies on the $u$-$v$ path. We now prove a result analogous to this for core center and eccentric diameter of $T$.
\begin{proposition}\label{prop1}
Let $T$ be a tree and let $u\in \mathfrak{C}(T)$. If $\epsilon^d(T)=f_T(x,y)$ for some pendant vertices $x$ and $y$ then $u$ lies on the $x$-$y$ path. 
\end{proposition}
\begin{proof}
Let $v_1,v_2,\ldots, v_k$ be the neighbours of $u$ in $T$. Let $B_i$ be the component of $T-u$ containing $v_i$ for $1\leq i \leq k$. By Theorem \ref{zha1}(iii), $\epsilon(v_i)=\min\{f_T(v_i,x), f_T(v_i,y)\}$ for $1\leq i \leq k$.

Suppose the path from $x$ to $y$ does not contain $u$. Then the $x$-$y$ path lies entirely in one of the component of $T-u$. Without loss of generality, let  $\ x,y\in V(B_1)$. Then $\epsilon(u)=min\{f_T(u,x), f_T(u,y)\}=min\{f_T(u,v_1,x), f_T(u,v_1,y)\}< min\{f_T(v_1,x), f_T(v_1,y)\}=\epsilon (v_1)$, a contradiction to $u\in \mathfrak{C}(T)$. So $u$ lies on the $x$-$y$ path. 
\end{proof}

Note that for a tree $T$ with $\epsilon^d(T)=f_T(x,y)$,  the length of the $x$-$y$ path may not be always equal to $diam(T)$. We have shown this in the next example.
\begin{figure}[h]
    \centering
  \begin{tikzpicture}[scale=1.5]
    \node[draw,circle,fill=black,scale=0.3](1)at (0,0){};
     \node[draw,circle,fill=black,scale=0.3](2)at (1,0){};
    \node[draw,circle,fill=black,scale=0.3](3)at (2,0){};
     \node[draw,circle,fill=black,scale=0.3](4)at (3,0){};
    \node[draw,circle,fill=black,scale=0.3](5)at (4,0){};
     \node[draw,circle,fill=black,scale=0.3](6)at (5,0){};
     \node[draw,circle,fill=black,scale=0.3](7)at (6,0){};
     \node[draw,circle,fill=black,scale=0.3](8)at (2.5,0.5){};
     \node[draw,circle,fill=black,scale=0.3](9)at (2,1){};
     \node[draw,circle,fill=black,scale=0.3](10)at (1.5,1.5){};
     \node[draw,circle,fill=black,scale=0.3](11)at (0.9,2.1){};
     \node[draw,circle,fill=black,scale=0.3](12)at (2.5,1.2){};
     \node[draw,circle,fill=black,scale=0.3](13)at (2,1.8){};
     \node[draw,circle,fill=black,scale=0.3](14)at (1.4,2.2){};
    \draw(1)--(2)--(3)--(4)--(5)--(6)--(7);
    \draw(4)--(8)--(9)--(10)--(11);
    \draw(8)--(12);
    \draw(9)--(13);
    \draw(10)--(14);

         \foreach \i in {1,...,11}
             \node[anchor=north] at (\i) {\i};
         \node[anchor=south] at (12) {12};
         \node[anchor=south] at (13) {13};
         \node[anchor=south] at (14) {14};

\end{tikzpicture} 
\caption{A tree $T$ with eccentric diameter distance is not equal to diam($T$)}\label{fig1}
\end{figure}
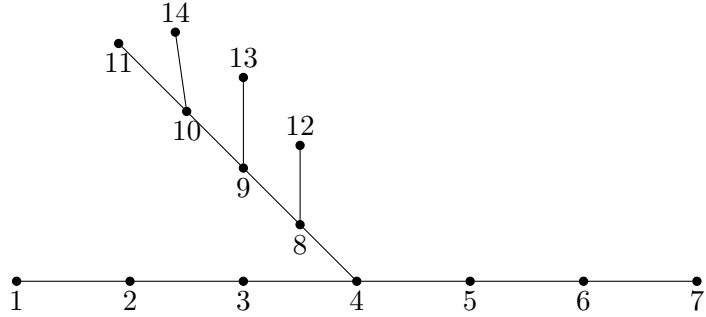
\begin{example}
Consider the tree $T$ in Figure \ref{fig1}. Then $f_T(1,7)=23$ and $f_T(7,11)=f_T(7,14)=f_T(1,11)=f_T(1,14)=32$, $f_T(1,13)=f_T(7,13)=40$, $f_T(1,12)=f_T(7,12)=44$, $f_T(12,13)=85$, $f_T(11,12)=f_T(12,14)=68$, $f_T(13,14)=f_T(11,13)=70$, $f_T(11,14)=71$.   So, $\epsilon^d(T)=f_T(1,7)$. Here  $d(1,7)=6$ but $diam(T)=d(1,11)=7$.
\end{example} 
 
\begin{lemma}\label{lem1}
Let $T$ be a tree with $u\in \mathfrak{C}(T)$ and $\epsilon^d(T)=f_T(p,q)$. Suppose $x\neq u$ is a vertex on the $u$-$p$ path. Then $\epsilon(x)=f_T(x,q).$ 
\end{lemma}
\begin{proof}
Suppose $P:v_0v_1v_2\cdots  v_k$ be the $u$-$p$ path where $u=v_0$ and $p=v_k$. Then by Theorem \ref{zha} and Corollary \ref{cor1}, $\epsilon(u)\geq  \epsilon(v_1)>\epsilon(v_2)>\dots >\epsilon(p)$. By Theorem \ref{zha1}(iii), $\epsilon(v_j)=\min\{f_T(v_j,p),f_T(v_j,q)\}$ for $1\leq j \leq k$. 

Assume that $\epsilon(v_j)=f_T(v_j,p)$. Then $\epsilon(v_{j-1})\leq f_T(v_{j-1},p)=f_T(v_{j-1},v_j,p)<f_T(v_j,p)=\epsilon(v_j)$,  a contradiction. So, $\epsilon(v_j)=f_T(v_j,q)$ for $1\leq j \leq k$.
\end{proof}

\begin{corollary}\label{cor2}
Let $T$ be a tree with $\epsilon^d(T)= f_T(p,q)$. Let $u$ and $v$ be two adjacent vertices on the $p$-$q$ path with $\epsilon(u)=f_T(u ,q)\neq f_T(u ,p) $ and $\epsilon(v)=f_T(v,p)\neq f_T(v,q)$. Then $\mathfrak{C}(T)=\{u,v\}\ or\ \{u\}\ or \{v\}$.
\end{corollary}
\begin{proof}
Let $x\in \mathfrak{C}(T)$. Then by Proposition \ref{prop1}, $x$ lies in the $p$-$q$ path. Suppose $x\notin \{u,v\}$. Then by Lemma \ref{lem1}, either $\epsilon(u)=f_T(u ,p)$ and $\epsilon(v)=f_T(v,p)$ or $\epsilon(u)=f_T(u ,q)$ and $\epsilon(v)=f_T(v,q)$, a contradiction to the given condition. So $\mathfrak{C}(T)\subseteq \{u,v\}$ and this proves the result.
\end{proof}

\section{Path-star trees}\label{pst}

Let $g$ and $n$ be two positive integers with $1\leq g\leq n-2$. A path-star tree on $n$ vertices is obtained from the  path $P_{n-g}:12\cdots n-g$ by adding $g$ pendant vertices to the vertex $n-g$ (see Figure \ref{fig3}). We denote it by $P_{n-g,g}$. Note that for $g=1$, the path-star tree $P_{n-1,1}$ is the path $P_n$ and for $g=n-2$, the path-star tree $P_{2,n-2}$ is the star $K_{1,n-1}$. The vertex numbered $1$ (see Figure \ref{fig3}) is called the initial vertex of $P_{n-g,g}$. For more details on path-star trees (or broom trees), we refer \cite{Na, Des, Klp}.

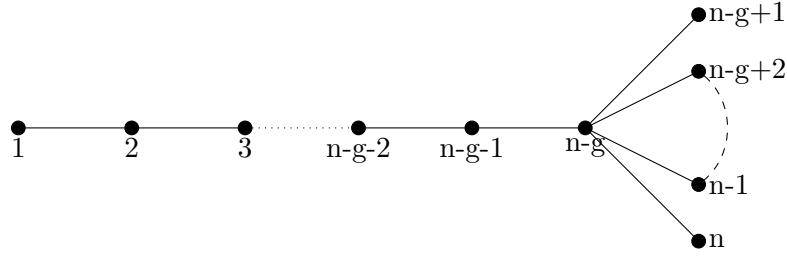
\begin{figure}[h]
    \centering
   \begin{tikzpicture}[scale=1.5]
    \node[draw,circle,fill=black,scale=0.5](1)at (0,0){};
    \node[draw,circle,fill=black,scale=0.5](2)at (1,0){};
    \node[draw,circle,fill=black,scale=0.5](3)at (2,0){};
    \node[draw,circle,fill=black,scale=0.5](ng2)at (3,0){};
    \node[draw,circle,fill=black,scale=0.5](ng1)at (4,0){};
    \node[draw,circle,fill=black,scale=0.5](ng)at (5,0){};
    \node[draw,circle,fill=black,scale=0.5](s1)at (6,1){};
    \node[draw,circle,fill=black,scale=0.5](s2)at (6,0.5){};
    \node[draw,circle,fill=black,scale=0.5](s3)at (6,-0.5){};
    \node[draw,circle,fill=black,scale=0.5](s4)at (6,-1){};
    \draw(1)--(2)--(3);
    \draw[dotted](3)--(ng2);
    \draw(ng2)--(ng1)--(ng);
    \draw(ng)--(s1);
    \draw(ng)--(s2);
    \draw(ng)--(s3);
    \draw(ng)--(s4);
    \node[anchor=north] at (1) {1};
    \node[anchor=north] at (2) {2};
    \node[anchor=north] at (3) {3};
    \node[anchor=north] at (ng2) {n-g-2};
    \node[anchor=north] at (ng1) {n-g-1};
    \node[anchor=north] at (ng) {n-g};
    \node[anchor=west] at (s1) {n-g+1};
    \node[anchor=west] at (s2) {n-g+2};
    \node[anchor=west] at (s3) {n-1};
    \node[anchor=west] at (s4) {n};
     \draw[dashed] (s2)to[<->,out=-40, in=40](s3);
\end{tikzpicture}
\caption{Path-star tree $P_{n-g,g}$}\label{fig3}
\end{figure}

Path-star trees play an important role in maximizing the pairwise distances between different central parts over trees on $n$ vertices. We  now  find the exact location of the core center of a path-star tree $P_{n-g,g}$. For $g\in\{1,n-2\}$, the position of core center of $P_{n-g,g}$ is known. So, we consider $2\leq g \leq n-3$. The core center of $P_{n-g,g}$ lie on the path from $2$ to $n-g$ as $\mathfrak{C}(P_{n-g,n})$ does not contain any pendant vertex. 

Let $T=P_{n-g,g}$. Then for $2\leq i\leq n-g$, $\epsilon(i)=min\{f_T(i,1),f_T(i, n-g+1)\}$. We have 
$$f_T(i,1)=f_T(i,1,n-g)+f_T(i,1,\overline{n-g})=2^g+(n-g-i)$$ and 
\begin{align*}
  f_T(i,n-g+1)&=f_T(i,n-g+1,i-1)+f_T(i,n-g+1,\overline{i-1})\\
              &=(i-1)2^{(g-1)}+2^{(g-1)}=i2^{(g-1)}=2^g+(i-2)2^{(g-1)}
\end{align*}
So, 
\begin{align*}
f_T(i,1)-f_T(i,n-g+1)&=(n-g-i)-(i-2)2^{(g-1)}\\
                     &=(n-g+2^g)-i(1+2^{(g-1)}). 
\end{align*}
 Thus $f_T(i,1)\leq f_T(i,n-g+1)$ if and only if $i \geq \frac{n-g+2^g}{1+2^{(g-1)}}$. Therefore,

$$\epsilon(i)= \left\{
\begin{array}{ll}
    
    f_T(i,n-g+1), & \text{if } \; 2\leq i \leq \lceil \frac{n-g+2^g}{1+2^{(g-1)}} \rceil-1,\\
    f_T(i,1), & \text{if }\; \lceil \frac{n-g+2^g}{1+2^{(g-1)}} \rceil  \leq i \leq n-g .\\
\end{array}
\right.$$

Let $p= \lceil \frac{n-g+2^g}{1+2^{(g-1)}} \rceil.$ For $2\leq i \leq n-g$, $f_T(i,n-g+1)=i2^{(g-1)}$ is a strictly increasing function and $f_T(i,1)=2^g+(n-g-i) $ is a strictly decreasing function. So $\mathfrak{C}(P_{n,n-g})\subseteq \{p-1,p\}.$

We have $p-1< \frac{n-g+2^g}{1+2^{(g-1)}}\leq p.$ This implies 
\begin{align}\label{p1}
   p(1+2^{(g-1)})<n-g+2^g+(1+2^{(g-1)}). 
\end{align}
Now 
\begin{align*}
  \epsilon(p)-\epsilon(p-1)&= 2^g+(n-g-p)-\{2^g+(p-3)2^{(g-1)}\}\\
                           &=n-g-p-(p-3)2^{(g-1)}\\
                           &=n-g+3\times 2^{(g-1)}-p(1+2^{(g-1)})\\
                           &>n-g+3\times 2^{(g-1)}-n+g-2^g-(1+2^{(g-1)})\;\;\;\mbox{(using (\ref{p1}))}\\
                           &=2\times 2^{(g-1)}-2^g-1\\
                           &=-1.
\end{align*}
This implies $\epsilon(p)-\epsilon(p-1)\geq 0$ as $\epsilon(p)-\epsilon(p-1)$ is an integer. So, $\mathfrak{C}(P_{n,n-g})=\{p\}$ or $\{p-1,p\}$.

From the above, we have $ \epsilon(p)-\epsilon(p-1)=n-g+3\times 2^{(g-1)}-p(1+2^{(g-1)})$. Thus 
$\epsilon(p)=\epsilon(p-1)$ if and only if $p=\frac{n-g+3\times 2^{(g-1)}}{(1+2^{(g-1)})}$. Therefore, if $p=\frac{n-g+3\times 2^{(g-1)}}{(1+2^{(g-1)})}$ then $\frac{n-g+3\times 2^{(g-1)}}{(1+2^{(g-1)})}$ is an integer. Also, if $\frac{n-g+3\times 2^{(g-1)}}{(1+2^{(g-1)})}$ is an integer, then $\frac{n-g+3\times2^{(g-1)}}{(1+2^{(g-1)})}=  \frac{n-g+2^{g}}{(1+2^{(g-1)})}+\frac{2^{(g-1)}}{(1+2^{(g-1)})}= \lceil \frac{n-g+2^g}{1+2^{(g-1)}}\rceil$ as $\frac{2^{(g-1)}}{(1+2^{(g-1)})}\in(0,1)$. Hence $\epsilon(p)=\epsilon(p-1)$ if and only if $\frac{n-g+3\times 2^{(g-1)}}{(1+2^{(g-1)})}$ is an integer. Thus we have the following.

\begin{lemma}\label{lem2.1}
The core center of the path-star tree $P_{n-g,g}$ is given by 
\begin{align*}
    \mathfrak{C}(P_{n,n-g})=\begin{cases}
   \{p-1,p\},& \text{if}\;\;\; \frac{n-g+3\times 2^{(g-1)}}{(1+2^{(g-1)})}\in \mathbb{N},\\
   \{p\},& otherwise,
    \end{cases}
\end{align*}
where  $p=\lceil \frac{n-g+2^g}{1+2^{(g-1)}} \rceil.$
\end{lemma}
\begin{remark}\label{rem:p}
Since $n-g\geq 3$, we have $\frac{n-g+2^g}{1+2^{(g-1)}}\geq \frac{3+2^g}{1+2^{(g-1)}}=\frac{1+2(1+2^{(g-1)})}{1+2^{(g-1)}}=2+\frac{1}{1+2^{(g-1)}}>2$. So, $p=\lceil \frac{n-g+2^g}{1+2^{(g-1)}} \rceil\geq 3.$
\end{remark}

The center, centroid and subtree core of path-star trees are obtained in \cite{Des}.
\begin{lemma}[\cite{Des}, Theorem 2.5]\label{lem2.2}
The center of the path-star tree $P_{n-g, g}$ is given by
\begin{equation*}
C(P_{n - g, g}) =
\begin{cases}
\left\{\frac{n - g + 2}{2}\right\}, &\text{if $n - g$ is even,}\\
\left\{\frac{n - g + 1}{2}, \frac{n - g + 3}{2}\right\}, &\text{if $n - g$ is odd.}
\end{cases}
\end{equation*}
\end{lemma}

\begin{lemma}[\cite{Des}, Theorem 3.3]\label{lem2.3}
The centroid of the path-star tree $P_{n - g, g}$ is given by
\begin{equation*}
C_d(P_{n - g, g}) =
\begin{cases}
\begin{cases}
\{\frac{n + 1}{2}\}, &\text{if $g \leq \frac{n - 1}{2}$}\\
\{n - g\}, &\text{if $g > \frac{n - 1}{2}$}
\end{cases}, &\text{if $n$ is odd,}\\
\begin{cases}
\{\frac{n}{2}, \frac{n}{2} + 1\}, &\text{if $g \leq \frac{n}{2} - 1$}\\
\{n - g\}, &\text{if $g > \frac{n}{2} - 1$}
\end{cases}, &\text{if $n$ is even.}
\end{cases}
\end{equation*}
\end{lemma}

\begin{lemma}[\cite{Des}, Theorem 2.4]\label{lem2.4}
The subtree core of the path-star tree $P_{n-g, g}$ is given by
\begin{equation*}
S_c(P_{n-g, g}) =
\begin{cases}
\begin{cases}
\left\{\frac{n - g + 2^g}{2}\right\}, &\text{if $n - g$ is even}\\
\left\{\frac{n - g - 1 + 2^g}{2}, \frac{n - g + 1 + 2^g}{2}\right\} , &\text{if $n - g$ is odd}
\end{cases}, &\text{if $2^g + 1 \leq n - g$,}\\
\{n - g\}, &\text{if $2^g + 1 > n - g.$}
\end{cases}
\end{equation*}
\end{lemma}

The  characteristic center of $P_{n-g,g}$ is not known for every $g$. In \cite{Pandey2}, the authors have conjectured that 
$|\chi_c(P_{n-g,g})|=2$. In $P_{n-g,g}$, the center, centroid, characteristic center, subtree core and core center  all lie on the path from $2$ to $n-g$. The order in which the center, centroid, characteristic center and  subtree core occur in $P_{n-g,g}$ are studied earlier. In $P_{n-g,g}$, $\chi_c(P_{n-g,g})$ lies on the path from $C(P_{n-g,g})$ to $C_d(P_{n-g,g})$ (see \cite{Klp}, Theorem 3.3) and $C_d(P_{n-g,g})$ lies on the path from $C(P_{n-g,g})$ to $S_c(P_{n-g,g})$ (see \cite{Des}, Proposition 4.1). In the next result we obtain the  position of core center with respect to other central parts in $P_{n-g,g}$.

\begin{proposition}\label{prop3}
Let $n\geq 5$ and $2\leq g \leq n-3$. Then the core center of $P_{n-g,g}$ lies on the path from  $2$ to  $C(P_{n-g,g})$.
\end{proposition}
\begin{proof}
We have 
\begin{align*}
2(n-g+2^g)&= 2(n-g+2)+2(2^g-2)\\
          &=2(n-g+2)+4(2^{g-1}-1)\\
          &<2(n-g+2)+(n-g+2)(2^{g-1}-1)\;\;\;\mbox{(as $4<n-g+2$)}\\
          &=(n-g+2)(1+2^{g-1})
\end{align*}
So, $\frac{n-g+2^g}{1+2^{(g-1)}}<\frac{n-g+2}{2}.$ Thus, if $n-g$ is even then $\frac{n-g+2}{2}$ is an integer and so $\lceil \frac{n-g+2^g}{1+2^{(g-1)}} \rceil\leq \frac{n-g+2}{2}$. If $n-g$ is odd, then  $\lceil \frac{n-g+2}{2} \rceil=\frac{n-g+3}{2}$ and so $\lceil \frac{n-g+2^g}{1+2^{(g-1)}} \rceil\leq \frac{n-g+3}{2}$. Now the result follows from Lemma \ref{lem2.1}, Remark \ref{rem:p} and Lemma \ref{lem2.2}.
\end{proof}

\section{Position of core center}\label{pcc}

In this section, we study the movement of the core center in  a tree by applying some perturbations to it. It can be easily checked that if $P:v_1v_2\cdots v_k$ be the $v_1$-$v_k$ path of a tree $T$ on $n$ vertices then $f_T(v_1,v_k)\leq 2^{n-k}$ and the equality happens if and only if all the vertices not on the $v_1$-$v_k$ path are pendant vertices of $T$. This fact is useful in proving some of the main results of this section.

\begin{lemma}\label{lem3}
Let $x$ be a pendant vertex of a tree $T$ and let $v_1\in V(T)$. Suppose $P:v_1v_2\cdots v_dx$ is the $v_1$-$x$ path in $T$. Then $(d+1)f_T(v_1,x)-f_T(v_1)\geq 0$   
\end{lemma}
\begin{proof}
Let $X_1$ be the component of $T-v_1v_2$ containing $v_1$, $X_2$ be the component of  $T-v_1v_2-v_2v_3$ containing $v_2$, $\ldots$ , $X_d$ be the component of  $T-v_{d-1}v_d-v_dx$ containing $v_d$ (see Figure \ref{fig2}). Then
\begin{figure}[h]
    \centering
 \begin{tikzpicture}[scale=1.2, every node/.style={scale=1}]
    \node[draw,circle,fill=black, inner sep=1.5pt, label=above:\(v_1\)](m)at (1.5,0){};
    \node[draw,circle,fill=black, inner sep=1.5pt, label=above:\(v_2\)](m1)at (4.5,0){};
    \node[draw,circle,fill=black, inner sep=1.5pt, label=above:\(v_d\)](m2)at (7.5,0){};
    \node[draw,circle,fill=black, inner sep=1.5pt, label=above:\(x\)](r)at (9,0){};
    \draw(m)--(m1);
    \draw[dashed](m1)--(m2);
    \draw(m2)--(r);
    \draw(1.5,0)--(1,-1)--(2,-1)--cycle;
    \node at (1.5,-0.8){\(X_1\)};
    \draw(4.5,0)--(4,-1)--(5,-1)--cycle;
    \node at (4.5,-0.8){\(X_2\)};
    \draw(7.5,0)--(7,-1)--(8,-1)--cycle;
    \node at (7.5,-0.8){\(X_d\)};
    \end{tikzpicture}
    \caption{The tree T}\label{fig2}
    
\end{figure}
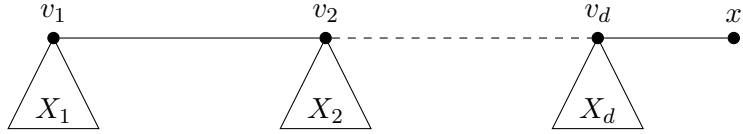
 
\begin{align*}
    f_T(v_1,x)=\prod_{i=1}^{d} f_{X_i}(v_i),\;\; f_T(v_1,v_d,\overline{x})=\prod_{i=1}^{d} f_{X_i}(v_i)
\end{align*} 
and for $1\leq j \leq d-1$
\begin{align*}
   f_T(v_1,v_j,\overline{v_{j+1}})= \prod_{i=1}^{j} f_{X_i}(v_i).
\end{align*}
So, $f_T(v_1)=f_{X_1}(v_1)+f_{X_1}(v_1)f_{X_2}(v_2)+\cdots+2f_{X_1}(v_1)f_{X_2}(v_2)\cdots f_{X_d}(v_d)$. 

Thus, $(d+1)f_T(v_1,x)-f_T(v_1)=(\prod_{i=1}^{d} f_{X_i}(v_i)-f_{X_1}(v_1))+(\prod_{i=1}^{d} f_{X_i}(v_i)-(f_{X_1}(v_1)f_{X_2}(v_2)))+\cdots+2(\prod_{i=1}^{d} f_{X_i}(v_i)-\prod_{i=1}^{d} f_{X_i}(v_i))\geq 0.$
\end{proof}

\begin{lemma} \label{lem4}
Let $v_k$ be a pendant vertex of a tree $T$ and let $P:v_1v_2\cdots v_k \;(k\geq 2)$ be the $v_1$-$v_k$ path in $T$. For $i\geq 2$, consider the vertex $v_i$ on the path $P$ and construct a new tree $T'$ from $T$ by deleting all the vertices of $T$ except the vertices that lie on the path from $v_1$ to $v_i$ and adding the same number of pendant vertices to the vertex $v_{i-1}$. Then $\frac{f_T(v_1,v_k)}{f_T(v_1)} \leq \frac{f_{T'}(v_1,v_k)}{f_{T'}(v_1)}$.
\end{lemma}
\begin{proof}
 Let $X_1$ be the component of $T-v_1v_2$ containing $v_1$, $X_2$ be the component of  $T-v_1v_2-v_2v_3$ containing $v_2$, $\ldots$ , $X_{k-1}$ be the component of  $T-v_{k-2}v_{k-1}-v_{k-1}v_k$ containing $v_{k-1}$. Then 
\begin{align*}
    f_T(v_1,v_k)=\prod_{j=1}^{k-1} f_{X_j}(v_j)
\end{align*}
and
\begin{align*}
     f_T(v_1)=f_{X_1}(v_1)+f_{X_1}(v_1)f_{X_2}(v_2)+\cdots+2f_{X_1}(v_1)f_{X_2}(v_2)\cdots f_{X_{k-1}}(v_{k-1}).
\end{align*}
Let $|V(T)|=n$. Note that in $T'$, $v_k$ is a pendant vertex adjacent to $v_{i-1}$. Then
\begin{align}\label{k1}
  f_{T'}(v_1,v_k)=2^{n-i} \ \mbox{and} \ 
  f_{T'}(v_1)=(i-2)+2^{n-i+1}.  
\end{align}
We have $$\frac{f_T(v_1,v_k)}{f_T(v_1)} - \frac{f_{T'}(v_1,v_k)}{f_{T'}(v_1)}=\frac{f_{T'}(v_1)f_T(v_1,v_k)-f_{T'}(v_1,v_k)f_T(v_1)}{f_T(v_1)f_{T'}(v_1)}.$$ So, it is enough to show that $f_{T'}(v_1)f_T(v_1,v_k)-f_{T'}(v_1,v_k)f_T(v_1)\leq 0$.

For $1\leq l \leq k-1$, suppose $M_l=\prod_{j=1}^{l} f_{X_j}(v_j)$. Then $f_T(v_1,v_k)=M_{k-1}$ and $f_T(v_1)=M_1+M_2+\cdots+2M_{k-1}.$  Since $M_l\geq 1$ for $1\leq l \leq k-1$,  we have

\begin{align}\label{k2}
(M_1+M_2+\cdots+M_{k-2})\geq (k-2)\geq (i-2),
\end{align}
and
\begin{align}\label{k3}
 M_{k-1}\leq 2^{n-k}\leq 2^{n-i} \ \mbox{as} \ k\geq i.
\end{align}

Thus
\begin{align*} 
&f_{T'}(v_1)f_T(v_1,v_k)-f_{T'}(v_1,v_k)f_T(v_1)\\
&=((i-2)+2^{n-i+1})M_{k-1}- 2^{n-i}(M_1+M_2+\cdots+2M_{k-1})\;\;\;\mbox{(using (\ref{k1}))}\\
&=(i-2)M_{k-1}+2^{n-i+1}M_{k-1}-2^{n-i}(M_1+\cdots+M_{k-2})-2^{n-i+1}M_{k-1}\\
&=(i-2)M_{k-1}-2^{n-i}(M_1+M_2+\cdots+M_{k-2})\\
& \leq (i-2)[M_{k-1}-2^{n-i}]\;\;\;\mbox{(using (\ref{k2}))}\\
&\leq 0 \;\;\;\mbox{(using (\ref{k3}))}.
\end{align*} 
This proves the result.
\end{proof}

The next two propositions are very useful in obtaining a tree which maximizes the distance between the core center and other central parts over trees on $n$ vertices.
\begin{proposition} \label{prop4}
Let $uv$ be an edge of a tree $T$ with $\epsilon_{T}(u)>\epsilon_{T}(v)$. Let $X$ and $Y$ be the components of $T-uv$ containing $u$ and $v$, respectively. Let $r\in V(Y)$ be a pendant vertex of $T$ and let $P$ be the $v$-$r$ path in $Y$. Let $y$ be a vertex on $P$ different from $v$. Construct a new tree $T'$ from $T$ by deleting all the vertices of $Y$ except the vertices that lie on the $v$-$y$ path and adding the same number of pendant vertices to the neighbour of $y$ on the $v$-$y$ path. Then $\epsilon_{T'}(u)>\epsilon_{T'}(v)$.   
\end{proposition}
\begin{proof}
Let $\epsilon_{T}(u)=f_T(u,p)$ and $\epsilon_{T}(v)=f_T(v,s)$. Then by Lemma \ref{lem2}, $s\in V(X).$ Let $X'$ and $Y'$ be the components of $T'-uv$ containing $u$ and $v$, respectively. Clearly $X=X'$.

Suppose $\epsilon_{T'}(u)=f_{T'}(u,w)$. If $w\in V(X')$ then ${\epsilon}_{T'}(v)\leq f_{T'}(v,w)=f_{T'}(v,u,w)<f_{T'}(u,w)={\epsilon}_{T'}(u)$.
So,  ${\epsilon}_{T'}(v)< {\epsilon}_{T'}(u)$.

If $w\in V(Y')$ then $\epsilon_{T'}(u)=f_{T'}(u,w)=f_{T'}(u,r)$ (since $w$ and $r$ are pendant vertices adjacent to the same vertex). In $T$, we have

\begin{align*}
 f_X(u,s)f_Y(v)=f_T(v,s)= \epsilon_T(v)<\epsilon_T(u)=f_T(u,p)\leq f_T(u,r)=f_X(u)f_Y(v,r).
\end{align*}
So, $\frac{f_X(u,s)}{f_X(u)}<\frac{f_Y(v,r)}{f_Y(v)}$ and by Lemma \ref{lem4}, we get
\begin{align*}
  \frac{f_X(u,s)}{f_X(u)}<\frac{f_Y(v,r)}{f_Y(v)}  \leq \frac{f_{Y'}(v,r)}{f_{Y'}(v)}.
\end{align*}
 
Since $X=X'$, we have $f_{X'}(u,s)f_{Y'}(v)< f_{X'}(u)f_{Y'}(v,r)$ and hence $f_{T'}(v,s)<f_{T'}(u,r)$. Now, from  $\epsilon_{T'}(u)=f_{T'}(u,r)$ we get $\epsilon_{T'}(v)\leq f_{T'}(v,s)<f_{T'}(u,r)=\epsilon_{T'}(u)$. This proves the result.
\end{proof}

\begin{proposition} \label{prop5}
Let $uv$ be an edge of a tree $T$ with $\epsilon_{T}(u)>\epsilon_{T}(v)$ and let $\epsilon_{T}(v)=f_T(v,s)$. Let $X$ and $Y$ be the components of $T-uv$ containing $u$ and $v$, respectively. Here $s\in V(X)$.  Let $Y=P_{k-g,g}\;(g\geq 1)$, where $v$ is initial vertex. Let $r\in V(Y)$ be a pendant vertex of $T$ different from $v$. Construct a new tree $\hat{T}$ from $T$ by deleting all the vertices of $X$ except the vertices that lie on the $u$-$s$ path and adding the same number of pendant vertices to the neighbour of $r$. Then $\epsilon_{\hat{T}}(u)>\epsilon_{\hat{T}}(v)$.   
\end{proposition}
\begin{proof}
Let  $\hat{X}$ and $\hat{Y}$ be the components of $\hat{T}-uv$ containing $u$ and $v$, respectively. If $X=\hat{X}$ then $T=\hat{T}$ and hence the result follows. So, assume that $X\neq \hat{X}.$ Clearly, $\hat{Y}=P_{m-h,h}$ with $m>k$. Then $m-k=h-g$ and  so, $m-h=k-g$.

Let $\epsilon_{\hat{T}}(u)=f_{\hat{T}}(u,p)$. If $p\in \hat{X}$ then $\epsilon_{\hat{T}}(v)\leq f_{\hat{T}}(v,p)=f_{\hat{T}}(v,u,p)<f_{\hat{T}}(u,p)=\epsilon_{\hat{T}}(u)$, and the result follows. 

Otherwise, let $p\in \hat{Y}$. Then $\epsilon_{\hat{T}}(u)=f_{\hat{T}}(u,p)=f_{\hat{T}}(u,r)$ as $\hat{Y}=P_{m-h,h}$. In $T$, we have $$f_Y(v)f_X(u,s)=f_T(v,s)=\epsilon_{T}(v)<\epsilon_{T}(u)\leq f_{T}(u,r)=f_X(u)f_Y(v,r).$$ Thus,

\begin{align}\label{k4}
\frac{f_{X}(u,s)}{f_{X}(u)}<\frac{f_{Y}(v,r)}{f_{Y}(v)}.
\end{align}
We have $\frac{f_Y(v,r)}{f_Y(v)} - \frac{f_{\hat{Y}}(v,r)}{f_{\hat{Y}}(v)}=\frac{f_{\hat{Y}}(v)f_Y(v,r)-f_{\hat{Y}}(v,r)f_Y(v)}{f_Y(v)f_{\hat{Y}}(v)}.$ Now, 
\begin{align*} 
&f_{\hat{Y}}(v)f_Y(v,r)-f_{\hat{Y}}(v,r)f_Y(v)\\
&=((m-h-1)+2^h)2^{g-1}- 2^{h-1}(k-g-1+2^g)\\
&=2^{g-1}(m-h-1)-2^{h-1}(k-g-1)\\
&=2^{g-1}(k-g-1)-2^{h-1}(k-g-1)\;\;\;(\mbox{as $m-h=k-g$})\\
&< 0 \;\;\;(\mbox{as $g<h$}).
\end{align*} 

So, $\frac{f_Y(v,r)}{f_Y(v)} < \frac{f_{\hat{Y}}(v,r)}{f_{\hat{Y}}(v)}$ and by (\ref{k4}), we get 

\begin{align}\label{k5}
\frac{f_{X}(u,s)}{f_{X}(u)}<\frac{f_{\hat{Y}}(v,r)}{f_{\hat{Y}}(v)}.
\end{align}

Let the length of the $u$-$s$ path in $T$ be $d$. Then $f_{\hat{X}}(u,s)=1$ and $f_{\hat{X}}(u)=d+1$. Thus 
\begin{align*} 
&\frac{ f_{\hat{X}}(u,s)}{f_{\hat{X}}(u)}- \frac{f_{X}(u,s)}{f_{X}(u)}\\
&=\frac{1}{d+1}-\frac{f_{X}(u,s)}{f_{X}(u)}\\
&=\frac{f_{X}(u)-(d+1)f_{X}(u,s)}{(d+1)f_{X}(u)}\\
&\leq 0 \;\;\;(\mbox{by Lemma \ref{lem3}}).
\end{align*}
So, 
\begin{align}\label{k6}
\frac{ f_{\hat{X}}(u,s)}{f_{\hat{X}}(u)}\leq \frac{ f_{X}(u,s)}{f_{X}(u)}.      
\end{align}

Then by (\ref{k4}), (\ref{k5}) and (\ref{k6}), we get 
\begin{align*}
 \frac{ f_{\hat{X}}(u,s)}{f_{\hat{X}}(u)}\leq \frac{ f_{X}(u,s)}{f_{X}(u)}<\frac{f_{Y}(v,r)}{f_{Y}(v)}< \frac {f_{\hat{Y}}(v,r)}{f_{\hat{Y}}(v)}. 
\end{align*}

Therefore, 
\begin{align*} 
&f_{\hat{Y}}(v)f_{\hat{X}}(u,s)<f_{\hat{X}}(u)f_{\hat{Y}}(v,r)\\
&\Rightarrow f_{\hat{T}}(v,s)<f_{\hat{T}}(u,r)=\epsilon_{\hat{T}}(u)\\
&\Rightarrow {{\epsilon}}_{\hat{T}}(v)\leq f_{\hat{T}}(v,s)<f_{\hat{T}}(u,r)=\epsilon_{\hat{T}}(u).
\end{align*}
This proves the result.
                         
\end{proof}

\section{Distance between central parts}\label{dcp}

For a tree $T$, by $d_T(C, \mathfrak{C})$, $d_T(C_d, \mathfrak{C})$, $d_T(S_c, \mathfrak{C})$ and $d_T(\chi_c, \mathfrak{C})$ we mean the distance between center and core center, distance between centroid and core center, distance between subtree core and core center and distance between characteristic center and core center in $T$, respectively. Note that for $n\leq 4$, any tree on $n$ vertices is either a star or a path and for star or path all these central parts are the same. The trees $P_5$, $K_{1,4}$ and $P_{3,2}$ are the only three non-isomorphic trees on $5$ vertices. It can be easily verified that $C(P_{3,2})=\chi_c(P_{3,2})=\mathfrak{C}(P_{3,2})=\{ 2, 3\}$ and $C_d(P_{3,2})=S_c(P_{3,2})=\{ 3 \}$ and hence  $d_T(C, \mathfrak{C})=d_T(C_d, \mathfrak{C})=d_T(S_c, \mathfrak{C})=d_T(\chi_c, \mathfrak{C})=0$. So, we consider trees on $n\geq 6$ vertices. For $g\in \{2,3,\ldots,n-3\}$, we denote $p_g=\lceil \frac{n-g+2^g}{1+2^{(g-1)}} \rceil$.

In this section, we show that over trees on $n$ vertices, these distances are maximized by some path-star trees. We first consider the pair center and core center.

\subsection{Center and core center}

Recall that if $P$ is a longest path of a tree $T$ then $C(T)$ is same as the center of $P$. We start this sub-section with the following result.
\begin{theorem}\label{thm:cc1}
Let T be a tree on $n\geq 6$ vertices. Then  $d_T(C,\mathfrak{C})\leq d_{P_{n-g,g}}(C,\mathfrak{C})$ for some positive integer $g$.
\end{theorem}
\begin{proof}
Let $T$ be a tree of on $n\geq 6$ vertices. Let $\mathfrak{C}(T)=\{w,u\}$ and $C(T)=\{x,y\}$ where $w=u$ if $|\mathfrak{C}(T)|=1$ and $x=y$ if $|C(T)|=1$.  The result is obvious if $d_T(C,\mathfrak{C})=0$, so assume that  $d_T(C,\mathfrak{C})=d_T(u,x)\geq 1$.

Let $P:z_0z_1 \cdots z_k$ be a longest path in $T$ containing both $C(T)$ and $\mathfrak{C}(T)$. In $P$, let $z_0$ be closer to $\mathfrak{C}(T)$ than to $C(T)$. Let $z_i=u$ and $z_j=x$ where $1\leq i<j<k$(as neither center nor core center contain pendant vertices). Then $z_{i+1}\notin \mathfrak{C}(T)$ and so  $\epsilon_T(u)>\epsilon_T(z_{i+1})$. Here the center of $P$ is either same as $C(T)$ or move towards the vertex $z_k$.

Let $X$ and $Y$ be the components of $T-uz_{i+1}$ containing $u$ and $z_{i+1}$, respectively. Form  a new tree $T'$ from $T$ by deleting all the vertices of $Y$ except the vertices that lie on the $z_{i+1}$-$z_k$ path and adding the same number of pendant vertices to the vertex $z_{k-1}$. Note that if $Y\cong P_{m-h,h}\;(h\geq 1)$ with $z_{i+1}$ as the initial vertex then $T'=T.$ By Proposition \ref{prop4}, $\epsilon_{T'}(u)>\epsilon_{T'}(z_{i+1})$ and so $\mathfrak{C}(T)=\mathfrak{C}(T')$ or $\mathfrak{C}(T')$ moves away from the vertex $z_{i+1}$(in fact moves away from $z_k$). In $T'$, the center $C(T')$ lies on the path from $z_{i+1}$ to $z_k$. Since the center of $P$ is either same as $C(T)$ or move towards the vertex $z_k$, we have  either $C(T)=C(T')$ or $C(T')$ moves towards the vertex $z_k$. So, $d_{T'}(C,\mathfrak{C})\geq d_T(u,x)=d_T(C,\mathfrak{C})$.

Take the new tree $T'$ as $T$ and continue the same process, till the component $Y$ of $T-uz_{i+1}$ become a path-star tree with the initial vertex $z_{i+1}$. Now consider the component $Y$ of $T-uz_{i+1}$, which is a path-star tree . Let $v=z_{i+1}$ and $\epsilon_{T}(v)=f_T(v,s)$. Then by Lemma \ref{lem2}, $s$ is in $X$. If  the $u$-$s$ path is the whole component $X$, then $T$ is a path-star tree. Otherwise, construct a new tree $\hat{T}$ from $T$ by deleting all the vertices of $X$ except the vertices that lie on the $u$-$s$ path and adding the same number of pendant vertices to the vertex $z_{k-1}$. Now the new tree $\hat{T}$ is a path-star tree. By Proposition \ref{prop5}, $\epsilon_{\hat{T}}(u)>\epsilon_{\hat{T}}(v)$ and so $\mathfrak{C}(\hat{T})=\mathfrak{C}(T)$ or $\mathfrak{C}(\hat{T})$ moves away from the vertex $v$. If $d_T(s,z_k)=diam(T)$, then $C(T)=C(\hat{T})$, otherwise $C(\hat{T})$ moves towards the vertex $z_k$. So, $d_{\hat{T}}(C,\mathfrak{C})\geq d_T(C,\mathfrak{C})$, and this proves the result.
\end{proof}

Now to find the maximum  distance between center and core center over trees on $n\geq 6$ vertices, we need to study the behaviour of the core center in path-star trees $P_{n-g,g}$ for $2\leq g\leq n-3$. The following two lemmas address this matter.

\begin{lemma}\label{lem:cc1}
Let $n\geq 6$ be an integer. For $g\in \{2,3,\ldots,n-3\}$, let $p_g=\lceil \frac{n-g+2^g}{1+2^{(g-1)}} \rceil$. Then
\begin{enumerate}
\item[(i)] $p_g=3$ if  $2^{g-1}+g\geq n-3$;
\item[(ii)] $p_{g-1}- p_{g}\geq 1$ if  $2^{g-1}+g\leq n-3$.
\end{enumerate}
\end{lemma}
\begin{proof}
(i) By Remark \ref{rem:p}, $p_g\geq 3$ for $g\in \{2,\ldots,n-3\}$. If $2^{g-1}+g\geq n-3$, then we have $n-g+2^g\leq 3+2^{g-1}+2^g=3(1+2^{g-1})$. Thus $\frac{n-g+2^g}{1+2^{(g-1)}} \leq 3$ and the result follows.\\

(ii) We have two cases here.\\
\textbf{Case-I:}  $2^{g-1}+g= n-3$\\
In this case, by (i),  $p_g=3$. Also, 
\begin{align*}
 p_{g-1}=\lceil \frac{n-g+1+2^{g-1}}{1+2^{(g-2)}} \rceil=\lceil \frac{2^{g-1}+3+1+2^{g-1}}{1+2^{(g-2)}} \rceil=\lceil \frac{4+2^{g}}{1+2^{(g-2)}}\rceil=\lceil \frac{4(1+2^{g-2})}{1+2^{(g-2)}} \rceil=4.
\end{align*}
Hence $p_{g-1}-p_{g}=1$.\\

\textbf{Case-II:} $2^{g-1}+g< n-3$ \\
Let $q_g=\frac{n-g+2^g}{1+2^{(g-1)}}$. Then 
\begin{align*}
 q_{g-1}-q_g&= \frac{n-g+1+2^{g-1}}{1+2^{g-2}}-\frac{n-g+2^g}{1+2^{g-1}}\\
            &=\frac{(n-g)(1+2^{g-1})+(1+2^{g-1})^2-(n-g)(1+2^{g-2})-2^g(1+2^{g-2})}{(1+2^{g-2})(1+2^{g-1})}\\
            &=\frac{(n-g)(2^{g-1}-2^{g-2})+1}{(1+2^{g-2})(1+2^{g-1})}\\
            &=\frac{2^{g-2}(n-g)+1}{(1+2^{g-2})(1+2^{g-1})}\\
            &>\frac{2^{g-2}(3+2^{g-1})+1}{(1+2^{g-2})(1+2^{g-1})}\;\;\;(\mbox{as $2^{g-1}+g< n-3$})\\
            &=\frac{2^{g-2}(3+2^{g-1})+1}{(1+2^{g-1}+2^{g-2}+2^{g-1}2^{g-2})}\\
            &=\frac{2^{g-2}(3+2^{g-1})+1}{1+2^{g-2}(3+2^{g-1})}\\
            &=1.
\end{align*}
So $p_{g-1}-p_{g}=\lceil q_{g-1} \rceil-\lceil q_g \rceil \geq \lceil q_{g-1}-q_{g}\rceil-1\geq 1$ as $q_{g-1}-q_g > 1$. This proves the result.
\end{proof}

\begin{theorem}\label{thm:cc2}
Let T be a tree on $n\geq 6$ vertices. Let $g_0\geq 2$ be the smallest positive integer such that $2^{g_0-1}+g_0\geq n-3$. Then  $d_T(C,\mathfrak{C})\leq \lfloor \frac{n-g_0-4}{2} \rfloor$ and the equality happens if $T$ is isomorphic to $P_{n-g_0,g_0}$.
\end{theorem}
\begin{proof}
Let $T$ be a tree on $n\geq 6$ vertices. By Theorem \ref{thm:cc1}, $d_T(C,\mathfrak{C})\leq d_{P_{n-g,g}}(C,\mathfrak{C})$ for some $g\in \{2,3,\ldots,n-3\}$.  
Let $g_0\geq 2$ be the smallest positive integer such that $2^{g_0-1}+g_0\geq n-3$. Consider the path-star tree $P_{n-g_0,g_0}$. Then 
\begin{equation*}
C(P_{n - g_0, g_0}) =
\begin{cases}
\left\{\frac{n - g_0 + 2}{2}\right\}, &\text{if $n - g_0$ is even,}\\
\left\{\frac{n - g_0 + 1}{2}, \frac{n - g_0 + 3}{2}\right\}, &\text{if $n - g_0$ is odd}
\end{cases}
\end{equation*}
 and by Lemma \ref{lem2.1} and Lemma \ref{lem:cc1}(i), $\mathfrak{C}(P_{n-g_0,g_0})=\{2,3\}$ or $\{3\}$. Thus, we have  $d_{P_{n - g_0, g_0}}(C,\mathfrak{C})=\lfloor \frac{n - g_0 + 2}{2}\rfloor-3=\lfloor\frac{n - g_0 -4}{2}\rfloor$.
 
Let $k\geq 1$ be an integer such that $g_0+k\leq n-3$. Clearly $2^{g_0+k-1}+g_0+k> n-3$. By Lemma \ref{lem:cc1}(i), we have $d_{P_{n - g_0-k, g_0+k}}(C,\mathfrak{C})=\lfloor \frac{n - g_0-k + 2}{2}\rfloor-3=\lfloor\frac{n - g_0 -k-4}{2}\rfloor \leq \lfloor\frac{n - g_0 -4}{2}\rfloor= d_{P_{n - g_0, g_0}}(C,\mathfrak{C}).$

Now let $k'\geq 1$ be an integer such that $g_0-k'\geq 2$. Clearly $2^{g_0-k'-1}+g_0-k'< n-3$. By Lemma \ref{lem:cc1}, $ p_{g_0-k'}\geq k'+3$. Thus, we have $d_{P_{n - g_0+k', g_0-k'}}(C,\mathfrak{C})\leq \lfloor \frac{n - g_0+k' + 2}{2}\rfloor-(k'+3)=\lfloor\frac{n - g_0 -k'-4}{2}\rfloor \leq \lfloor\frac{n - g_0 -4}{2}\rfloor= d_{P_{n - g_0, g_0}}(C,\mathfrak{C}).$ 

Hence, among all path-star trees on $n\geq 6$ vertices, $P_{n - g_0, g_0}$ maximizes the distance between center and core center and $d_{P_{n - g_0, g_0}}(C,\mathfrak{C})= \lfloor \frac{n-g_0-4}{2} \rfloor$. This proves the result.
\end{proof}

\subsection{Centroid and core center}

Note that for a tree $T$ on $n$ vertices, $u\in C_d(T)$ if and only if $\omega(u)\leq \frac{n}{2}$ (see \cite{Ka}, Theorem 1). Also, if $C_d(T)=\{u,v\}$ then $n$ is even and $\omega(u)=\omega(v)=\frac{n}{2}$. We start this subsection with the following result.
\begin{theorem}\label{thm:cd1}
Let T be a tree on $n\geq 6$ vertices. Then  $d_T(C_d,\mathfrak{C})\leq d_{P_{n-g,g}}(C_d,\mathfrak{C})$ for some positive integer $g$.
\end{theorem}
\begin{proof}
Let $T$ be a tree of on $n\geq 6$ vertices. Let $\mathfrak{C}(T)=\{w,u\}$ and $C_d(T)=\{x,y\}$ where $w=u$ if $|\mathfrak{C}(T)|=1$ and $x=y$ if $|C_d(T)|=1$.  The result is obvious if $d_T(C_d,\mathfrak{C})=0$, so assume that  $d_T(C_d,\mathfrak{C})=d_T(u,x)\geq 1$.

We have $\omega_T(y)\leq \frac{n}{2}$ as $y\in C_d(T)$(here $y=x$ if $|C_d(T)|=1$). Let $B_1,B_2,\ldots,B_k$ be the branches of $T$ at $u$ and let the branch $B_1$ contains $C_d(T)$. Let $v$ be the vertex of $B_1$ adjacent to $u$. Clearly, $v\notin \mathfrak{C}(T)$ and so $\epsilon_T(u)>\epsilon_T(v)$. Let $X$ and $Y$ be the components of $T-uv$ containing $u$ and $v$, respectively. Let $r$ be a vertex in $Y$ adjacent to $y$ but not on the $v$-$y$ path. Form  a new tree $T'$ from $T$ by deleting all the vertices of $Y$ except the vertices that lie on the $v$-$r$ path and adding the same number of pendant vertices to the vertex $y$. Note that if $Y=P_{m-h,h}\;(h\geq 1)$ then $T'=T.$ By Proposition \ref{prop4}, $\epsilon_{T'}(u)>\epsilon_{T'}(v)$ and so $\mathfrak{C}(T)=\mathfrak{C}(T')$ or $\mathfrak{C}(T')$ moves away from the vertex $y$. In the new tree $T'$, $\omega_{T'}(y)\leq \frac{n}{2}$, so $y\in C_d(T')$. Thus,  $d_{T'}(C_d,\mathfrak{C})\geq d_T(u,x)=d_T(C_d,\mathfrak{C})$.

Take the new tree $T'$ as $T$ and continue the same process, till the component $Y$ of $T-uv$ become a path-star tree with $v$ as the initial vertex. Now consider the component $Y$ of $T-uv$, which is a path-star tree. Let  $\epsilon_{T}(v)=f_T(v,s)$. Then by Lemma \ref{lem2}, $s$ is in $X$. If  the $u$-$s$ path is the whole component $X$, then $T$ is a path-star tree. Otherwise, construct a new tree $\hat{T}$ from $T$ by deleting all the vertices of $X$ except the vertices that lie on the $u$-$s$ path and adding the same number of pendant vertices to the vertex $y$. Now the new tree $\hat{T}$ is a path-star tree. By Proposition \ref{prop5}, $\epsilon_{\hat{T}}(u)>\epsilon_{\hat{T}}(v)$ and so $\mathfrak{C}(\hat{T})=\mathfrak{C}(T)$ or $\mathfrak{C}(\hat{T})$ moves away from the vertex $y$. In the new tree $\hat{T}$, $\omega_{\hat{T}}(y)\leq \frac{n}{2}$, so $y\in C_d(\hat{T})$. Thus, $d_{\hat{T}}(C_d,\mathfrak{C})\geq d_T(C_d,\mathfrak{C})$ and this proves the result.
\end{proof}

\begin{lemma}\label{lem:cd}
Let $n\geq 3$ be an integer.  Then $2^{\lfloor\frac{n}{2}\rfloor-1}+\lfloor\frac{n}{2}\rfloor> n-3$.
\end{lemma}
\begin{proof}
We have 
$$2^{\lfloor\frac{n}{2}\rfloor-1}+\lfloor \frac{n}{2} \rfloor - n + 3= \left\{
\begin{array}{ll}
  2^{\frac{n-2}{2}}-\frac{n-2}{2}+2 ,  &\text{if $n$ is even,}\\
  2^{\frac{n-3}{2}}-\frac{n-3}{2}+1, &\text{if $n$ is odd.}\\
\end{array}
\right.$$
Since $2^x>x$ for all $x\geq 0$, the result follows.
\end{proof}
\begin{theorem}\label{thm:cd2}
Let T be a tree on $n\geq 6$ vertices. Let $g_0\geq 2$ be the smallest positive integer such that $2^{g_0-1}+g_0\geq n-3$. Then  $d_T(C_d,\mathfrak{C})\leq \lfloor \frac{n-5}{2} \rfloor$ and the equality happens if $T$ is isomorphic to $P_{n-g_0,g_0}$.
\end{theorem}
\begin{proof}
Let $T$ be a tree on $n\geq 6$ vertices. By Theorem \ref{thm:cd1}, $d_T(C_d,\mathfrak{C})\leq d_{P_{n-g,g}}(C_d,\mathfrak{C})$ for some $g\in \{2,3,\ldots,n-3\}$.  
Let $g_0\geq 2$ be the smallest positive integer such that $2^{g_0-1}+g_0\geq n-3$. Then by Lemma \ref{lem:cd}, $g_o\leq \lfloor \frac{n}{2} \rfloor$. Consider the path-star tree $P_{n-g_0,g_0}$. There are two cases based on whether $n$ is odd or even.

\textbf{Case-I:} $n$ is odd.\\
Since $g_o\leq \lfloor \frac{n}{2} \rfloor$, by Lemma \ref{lem2.3},  $C_d(P_{n-g_0,g_0})=\{\frac{n+1}{2}\}$. Also by Lemma \ref{lem2.1} and Lemma \ref{lem:cc1}(i), $\mathfrak{C}(P_{n-g_0,g_0})=\{2,3\}$ or $\{3\}$. Thus, $d_{P_{n - g_0, g_0}}(C_d,\mathfrak{C})= \frac{n+1}{2} - 3=\frac{n-5}{2}.$

Let $k\geq 1$ be an integer such that $g_0+k\leq n-3$. Clearly $2^{g_0+k-1}+g_0+k> n-3$. By Lemma \ref{lem2.3}, we have 
$$C_d(P_{n-g_0-k,g_0+k})=\left\{
\begin{array}{ll}
  \{\frac{n+1}{2}\},  &\text{if $g_0+k\leq \frac{n-1}{2}$,}\\
  \{n-g_0-k\}, &\text{if $g_0+k> \frac{n-1}{2}$.}\\
\end{array}
\right.$$ 
Then by Lemma \ref{lem:cc1}(i), $d_{P_{n - g_0-k, g_0+k}}(C_d,\mathfrak{C})\leq \frac{n+1}{2} - 3=\frac{n-5}{2}.$

Now let $k'\geq 1$ be an integer such that $g_0-k'\geq 2$. Clearly $2^{g_0-k'-1}+g_0-k'< n-3$. By Lemma \ref{lem2.3}, we have $C_d(P_{n-g_0+k',g_0-k'})=\{\frac{n+1}{2}\}$. By Lemma \ref{lem:cc1}, $ p_{g_0-k'}\geq k'+3$. Thus, we have $d_{P_{n - g_0+k', g_0-k'}}(C_d,\mathfrak{C})\leq \frac{n+1}{2}-k'-3<\frac{n-5}{2}=d_{P_{n - g_0, g_0}}(C_d,\mathfrak{C})$.\\

\textbf{Case-II:} $n$ is even.\\
Since $g_o\leq \lfloor \frac{n}{2} \rfloor$, by Lemma \ref{lem2.3}
$$C_d(P_{n-g_0,g_0})=\left\{
\begin{array}{ll}
  \{\frac{n}{2},\frac{n}{2}+1 \},  &\text{if $g_0\leq \frac{n}{2}-1$,}\\
  \{\frac{n}{2}\}, &\text{if $g_0 =\frac{n}{2}$.}\\
\end{array}
\right.$$ 
Thus, $d_{P_{n - g_0, g_0}}(C_d,\mathfrak{C})= \frac{n}{2} - 3=\frac{n-6}{2}=\lfloor \frac{n-5}{2} \rfloor.$

Let $k\geq 1$ be an integer such that $g_0+k\leq n-3$. By Lemma \ref{lem2.3}, we have 
$$C_d(P_{n-g_0-k,g_0+k})=\left\{
\begin{array}{ll}
 \{\frac{n}{2},\frac{n}{2}+1 \},  &\text{if $g_0+k\leq \frac{n}{2}-1$,}\\
  \{n-g_0-k\}, &\text{if $g_0+k> \frac{n}{2}-1$.}\\
\end{array}
\right.$$ 
Then by Lemma \ref{lem:cc1}(i), $d_{P_{n - g_0-k, g_0+k}}(C_d,\mathfrak{C})\leq \frac{n}{2} - 3=\frac{n-6}{2}.$

Now let $k'\geq 1$ be an integer such that $g_0-k'\geq 2$. By Lemma \ref{lem2.3}, we have $C_d(P_{n-g_0+k',g_0-k'})=\{\frac{n}{2},\frac{n}{2}+1 \}.$ Also, by Lemma \ref{lem:cc1}, $ p_{g_0-k'}\geq k'+3$. Thus, $d_{P_{n - g_0+k', g_0-k'}}(C_d,\mathfrak{C})\leq \frac{n}{2}-k'-3<\frac{n-6}{2}=d_{P_{n - g_0, g_0}}(C_d,\mathfrak{C})$. This proves the result.
\end{proof}

\subsection{Subtree core and core center}

The following lemma compares the subtree core of two trees and it is useful for tackling the problem of maximizing the distance between subtree core and core center over trees on $n$ vertices.

\begin{lemma}[\cite{Des}, Lemma 2.2]\label{lem:sc1}
Let $T$ be a tree, $v\in S_c(T)$ and $y$ be a pendant vertex of $T$ not adjacent to $v$. If $\tilde{T}$ is the tree obtained by detaching $y$ from $T$ and adding it as a pendant vertex adjacent to $v$, then $S_c(\tilde{T}) = \{v\}.$
\end{lemma}
\begin{theorem}\label{thm:sc1}
Let T be a tree on $n\geq 6$ vertices. Then  $d_T(S_c,\mathfrak{C})\leq d_{P_{n-g,g}}(S_c,\mathfrak{C})$ for some positive integer $g$.
\end{theorem}
\begin{proof}
Let $T$ be a tree of on $n\geq 6$ vertices. Let $\mathfrak{C}(T)=\{w,u\}$ and $S_c(T)=\{x,y\}$ where $w=u$ if $|\mathfrak{C}(T)|=1$ and $x=y$ if $|S_c(T)|=1$.  The result is obvious if $d_T(S_c,\mathfrak{C})=0$, so assume that  $d_T(S_c,\mathfrak{C})=d_T(u,x)\geq 1$.

Let $B_1,B_2,\ldots,B_k$ be the branches of $T$ at $u$ and let the branch $B_1$ contains $S_c(T).$ Let $v$ be the vertex of $B_1$ adjacent to $u$. Clearly, $v\notin \mathfrak{C}(T)$ and so $\epsilon_T(u)>\epsilon_T(v)$. Let $X$ and $Y$ be the components of $T-uv$ containing $u$ and $v$, respectively. Let $r$ be a vertex in $Y$ adjacent to $y$ but not on the $v$-$y$ path. Form  a new tree $T'$ from $T$ by deleting all the vertices of $Y$ except the vertices that lie on the $v$-$r$ path and adding the same number of pendant vertices to the vertex $y$. Note that if $Y=P_{m-h,h}\;(h\geq 1)$ then $T'=T.$ By Proposition \ref{prop4}, $\epsilon_{T'}(u)>\epsilon_{T'}(v)$ and so $\mathfrak{C}(T)=\mathfrak{C}(T')$ or $\mathfrak{C}(T')$ moves away from the vertex $y$. The new tree $T'$ can also be obtained from $T$ by deleting required pendant vertices step by step and adding them as pendant vertices to the vertex $y$. Then by Lemma \ref{lem:sc1}, $S_c(T')=\{y\}$ and thus,  $d_{T'}(S_c,\mathfrak{C})\geq d_T(u,x)=d_T(S_c,\mathfrak{C})$.

Take the new tree $T'$ as $T$ and continue the same process, till the component $Y$ of $T-uv$ become a path-star tree with $v$ as the initial vertex. Now consider the component $Y$ of $T-uv$ is a path-star tree. Let  $\epsilon_{T}(v)=f_T(v,s)$. Then by Lemma \ref{lem2}, $s$ is in $X$. If  the $u$-$s$ path is the whole component $X$, then $T$ is a path-star tree. Otherwise, construct a new tree 
$\hat{T}$ from $T$ by deleting all the vertices of $X$ except the vertices that lie on the $u$-$s$ path and adding the same number of pendant vertices to the vertex $y$. Now the new tree $\hat{T}$ is a path-star tree. By Proposition \ref{prop5}, $\epsilon_{\hat{T}}(u)>\epsilon_{\hat{T}}(v)$ and so $\mathfrak{C}(\hat{T})=\mathfrak{C}(T)$ or $\mathfrak{C}(\hat{T})$ moves away from the vertex $y$. As mentioned above, the new tree $\hat{T}$ can also be obtained from $T$ by deleting pendant vertices step by step and adding them as pendant vertices to the vertex $y$. Then by Lemma \ref{lem:sc1}, $S_c(\hat{T})=\{y\}$ and thus, $d_{\hat{T}}(S_c,\mathfrak{C})\geq d_T(S_c,\mathfrak{C})$ and this proves the result.
\end{proof}

\begin{lemma}\label{lem:sc2}
Let $n\geq 6$ and $g\geq 2$.  If  $2^{g-1}+g \geq n-3$ then $2^{g}+1\geq n-g$. Moreover, the equality happens if and only if $(n,g)=(7,2)$.
\end{lemma}
\begin{proof}
We have $2^{g}+1=2^{g-1}+2^{g-1}+g+1-g \geq 2+(2^{g-1}+g)+1-g\geq 2+n-3+1-g = n-g$ and the equality happens if and only if  $g=2$ and $2^{g-1}+g = n-3$. So, for  $2^{g-1}+g \geq n-3$,  $2^{g}+1= n-g$ if and only if $(n,g)=(7,2)$.
\end{proof}
\begin{theorem}\label{thm:sc2}
Let T be a tree on $n\geq 6$ vertices. Let $g_0\geq 2$ be the smallest positive integer such that $2^{g_0-1}+g_0\geq n-3$. Then 
$$d_T(S_c,\mathfrak{C})\leq\left\{
\begin{array}{ll}
 1,  &\text{if $n=7$,}\\
 n-g_0-3, &\text{if $n\neq 7$}\\
\end{array}
\right.$$ 
 and equality happen in both the cases if $T$ is isomorphic to $P_{n-g_0,g_0}$.
\end{theorem}
\begin{proof}
Let $T$ be a tree on $n\geq 6$ vertices. By Theorem \ref{thm:cd1}, $d_T(C_d,\mathfrak{C})\leq d_{P_{n-g,g}}(C_d,\mathfrak{C})$ for some $g\in \{2,3,\ldots,n-3\}$.
Let $g_0\geq 2$ be the smallest positive integer such that $2^{g_0-1}+g_0\geq n-3$. Consider the path-star tree $P_{n-g_0,g_0}$. There are two cases:

\textbf{Case-I:} $n=7$.\\
Then $g_0=2$ and by Lemma \ref{lem2.1} $\mathfrak{C}(P_{5,2})=3$. Since $2^{g_0}+1=5=n-g_0$, by Lemma \ref{lem2.4} $S_c(P_{5,2})=\{4,5\}$ and so, $d_{P_{5,2}}(S_c, \mathfrak{C})=1$. Using Lemma \ref{lem2.1} and Lemma \ref{lem2.4}, it can be checked that $d_{P_{4,3}}(S_c, \mathfrak{C})=1$ and $d_{P_{3,4}}(S_c, \mathfrak{C})=0$.

\textbf{Case-II:} $n\neq 7$.\\
By Lemma \ref{lem:sc2},  $2^{g_0-1}+g_0 \geq n-3$  implies $2^{g_0}+1> n-g_0$. Therefore, $\mathfrak{C}(P_{n-g_0,g_0})=\{2,3\}$ or $\{3\}$ and $S_c(P_{n-g_0,g_0})=\{n-g_0\}$, hence $d_{P_{n-g_0,g_0}}(S_c, \mathfrak{C})=n-g_0-3$.\\

Let $k\geq 1$ be an integer such that $g_0+k\leq n-3$. Then by Lemma \ref{lem:sc2}, $2^{g_0+k}+1> n-(g_0+k)$ and so, $S_c(P_{n-g_0-k,g_0+k})=\{n-g_0-k\}$. Hence $d_{P_{n-g_0-k,g_0+k}}(S_c, \mathfrak{C})=n-g_0-k-3<n-g_0-3=d_{P_{n-g_0,g_0}}(S_c, \mathfrak{C})$.

Now let $k'\geq 1$ be an integer such that $g_0-k'\geq 2$. Clearly $2^{g_0-k'-1}+g_0-k'< n-3$. By Lemma \ref{lem:cc1}, $ p_{g_0-k'}\geq k'+3$. Also, by Lemma \ref{lem2.4} $S_c(P_{n-g_0+k',g_0-k'})=\{a,b\}$ or $\{c\}$ where both $a$ and $c$ are less than or equal to $n-g_0+k'$. Hence $d_{P_{n-g_0+k',g_0-k'}}(S_c, \mathfrak{C})\leq (n-g_0+k')-(k'+3)=n-g_0-3=d_{P_{n-g_0,g_0}}(S_c, \mathfrak{C})$. This proves the result.
\end{proof}

\subsection{Characteristic center and core center}

Let $T$ be a tree and $v\in V(T)$. Let $C_1,C_2,\ldots, C_k$ be the connected components of $T-v$. For $i=1,2,\ldots,k$,  let $\hat{L}(C_i)$ denoted the principle sub-matrix of the Laplacian matrix $L(T)$ corresponding to the vertices  of $C_i$. Then $\hat{L}(C_i)$ is invertible  and $\hat{L}(C_i)^{-1}$ is a positive matrix, called the bottleneck matrix for $C_i$. 

By Perron-Frobenius theorem, $\hat{L}(C_i)^{-1}$ has simple dominant eigenvalue, called the Perron value of $C_i$ at $v$. The component $C_j$ is called a Perron component at $v$ if the Perron value of $C_j$ is maximal among the components  $C_1,C_2,\ldots, C_k$ at $v$. The entries of the bottleneck matrices is described in the next result.

\begin{lemma}[\cite{Kirk2}, Proposition 1]\label{lem:ch1}
Let $T$ be a tree and let $v \in V(T)$. Let $C_1$ be a component of $T-v$ and $L_1$ be the sub-matrix of $L(T)$ corresponding to $C_1$. Then $L_{1}^{-1}=(c_{ij})$, where $c_{ij}$ is the number of edges in common between the paths $P_{iv}$ and $P_{jv}$, where $P_{iv}$ denotes the path joining $i$ and $v$. 
\end{lemma}

Now we will discuss some results which help us to understand the movement of characteristic center under some perturbations.

\begin{theorem}[\cite{Kirk2}, Corollary 1.1]\label{thm:ch1}
Let $T$ be a tree. Then  $\chi_c(T)=\{v_i,v_j\}$  if and only if the component $T_i$ at vertex $v_j$ containing the vertex $v_i$ is the unique Perron component at $v_j$ while the component $T_j$ at vertex $v_i$ containing the vertex $v_j$ is the unique Perron component at $v_i$.
\end{theorem}

\begin{theorem}[\cite{Kirk2}, Corollary 2.1]\label{thm:ch2}
Let $T$ be a tree. Then $\chi_c(T)=\{ v \}$ if and only if there are two or more Perron components of $T$ at $v$.
\end{theorem} 

\begin{theorem}[\cite{Kirk2}, Proposition 2]\label{thm:ch3}
Let $T$ be a tree. If $v\notin \chi_c(T)$ then the unique Perron component at $v$ contains the characteristic center of $T$.
\end{theorem}

For a real square matrix $A$, we denote the spectral radius of $A$ by $\rho(A)$. For non-negative square matrices $A$ and $B$ with order of $B$ is greater or equal to order of A, by the notation $A\ll B$  we mean that there exists permutation matrices $P$ and $Q$ such that $P^TAP$ is entry wise dominated by a principal submatrix of $Q^TBQ$, with strict inequality in at least one place, in case $A$ and $B$ have same order. Using Perron-Frobenius theory, one can prove that if $B$ is irreducible with $A\ll B$  then $\rho(A)<\rho(B)$.  

\begin{theorem}\label{thm:ch4}
Let T be a tree on $n\geq 6$ vertices. Then  $d_T(\chi_c,\mathfrak{C})\leq d_{P_{n-g,g}}(\chi_c,\mathfrak{C})$ for some positive integer $g$.
\end{theorem}
\begin{proof}
Let $T$ be a tree of on $n\geq 6$ vertices. Let $\mathfrak{C}(T)=\{w,u\}$ and $\chi_c(T)=\{x,y\}$ where $w=u$ if $|\mathfrak{C}(T)|=1$ and $x=y$ if $|\chi_c(T)|=1$.  The result is obvious if $d_T(\chi_c,\mathfrak{C})=0$, so assume that  $d_T(\chi_c,\mathfrak{C})=d_T(u,x)\geq 1$.

Let  $C_1,C_2,\ldots,C_k$ be the connected components of $T-x$ and let $C_1$ be the component containing $u$. For $i\in \{2,\ldots,k\}$, let $\alpha_i=\max\{d_T(x,z)|z\in V(C_i)\}$. Let $\alpha=\max\{\alpha_i|2\leq i\leq k\}$ and suppose $\alpha=d_T(x,r)$. Then clearly $r$ is a pendant vertex of $T$ and without loss of generality, take $r\in V(C_2)$.

Let $B_1,B_2,\ldots,B_l$ be the branches of $T$ at $u$ and let the branch $B_1$ contains $\chi_c(T).$ Let $v$ be the vertex of $B_1$ adjacent to $u$. Clearly, $v\notin \mathfrak{C}(T)$ and so $\epsilon_T(u)>\epsilon_T(v)$. Let $X$ and $Y$ be the components of $T-uv$ containing $u$ and $v$, respectively. Then $r$ is in the component $Y$. Form  a new tree $T'$ from $T$ by deleting all the vertices of $Y$ except the vertices that lie on the $v$-$r$ path and adding the same number of pendant vertices to the vertex adjacent to $r$. By Proposition \ref{prop4}, $\epsilon_{T'}(u)>\epsilon_{T'}(v)$. 

Now in $T'$, let $X'$ and $Y'$ be the components of $T'-uv$ containing $u$ and $v$, respectively. Then $X=X'$ and $Y'$ is a path-star tree with initial vertex $v$. Let  $\epsilon_{T'}(v)=f_{T'}(v,s)$. Then by Lemma \ref{lem2}, $s$ is in $X'$. If  the $u$-$s$ path is the whole component $X$, then $T'$ is a path-star tree. Otherwise, construct a new tree 
$\hat{T}$ from $T'$ by deleting all the vertices of $X'$ except the vertices that lie on the $u$-$s$ path and adding the same number of pendant vertices to the vertex adjacent to $r$. Now the new tree $\hat{T}$ is a path-star tree. By Proposition \ref{prop5}, $\epsilon_{\hat{T}}(u)>\epsilon_{\hat{T}}(v)$ and so $\mathfrak{C}(\hat{T})=\mathfrak{C}(T)$ or $\mathfrak{C}(\hat{T})$ moves towards the vertex $s$(in fact move away from $x$).

We will now obtain $\hat{T}$ from $T$ by using some different perturbation. If $\chi_c(T)=\{x,y\}$, then by Theorem \ref{thm:ch1}, the component containing $y$ at $x$ is the Perron component and $y\notin V(C_1)$. Also, if $\chi_c(T)=\{x\}$ then by Theorem \ref{thm:ch2}, $T-x$ has atleast two Perron components. So in $T$, atleast one component among $C_2,\ldots,C_k$ is a Perron component at $x$. 

In $T$, the vertex $s$ is in the component $C_1$. Let $w'$ be the vertex of $C_1$ adjacent with $x$. Form a new tree $T_1$ from $T$  by deleting all the vertices of $C_1$ except the vertices that lie on the $w'$-$s$ path and adding the same number of pendant vertices to the vertex adjacent to $r$. Note that if $C_1$ is the $w'$-$s$ path then $T_1=T$. Now, let $D_1,D_2,\ldots,D_k$ be the connected components of $T_1-x$, where $D_i$ is the component corresponding to $C_i$ for $i=1,2,\ldots,k$. Here the $w'$-$s$ path is the component $D_1$. Then by Lemma \ref{lem:ch1}, $\hat{L}(D_1)^{-1}\ll\hat{L}(C_1)^{-1}$ and so $D_1$ is not a Perron component of $T_1$ at $x$. Thus, $\chi_c(T_1)=\chi_c(T)$ or $\chi_c(T_1)$ move away from $s$(in fact move away from $u$) which follows from Theorem \ref{thm:ch3}. 

Let $D\equiv \cup_{i=2}^kD_i$. Note that the vertex $r$ lies in the component $D_2$ and let $z'$ be the vertex of $D_2$ adjacent with $x$. Construct a new tree $T_2$  from $T_1$ by deleting all the vertices of $D$ except the vertices that lie on the $z'$-$r$ path and adding the same number of pendant vertices to the vertex adjacent to $r$. Clearly $T_2$ is a path-star tree and it is isomorphic to $\hat{T}$. Let $F_1$ and $F_2$ be the two components of $T_2$ at $x$ where $F_1$ is same as $D_1$. By Lemma \ref{lem:ch1}, $\hat{L}(D)^{-1}\ll\hat{L}(F_2)^{-1}$ and so $F_2$ is the Perron component of $T_2$ at $x$. So $\chi_c(\hat{T})=\chi_c(T_2)=\chi_c(T)$ or $\chi_c(T_2)$ move towards the vertex $r$(in fact move away from $u$) which follows from Theorem \ref{thm:ch3}.

Thus from the tree $T$, we construct a new tree $\hat{T}$ such that $\chi_c(\hat{T})=\chi_c(T)$ or $\chi_c(\hat{T})$ moves away from $u$ towards the vertex $r$. Also, $\mathfrak{C}(\hat{T})=\mathfrak{C}(T)$ or $\mathfrak{C}(\hat{T})$ moves away from $x$ towards the vertex $s$. So, $d_T(\chi_c,\mathfrak{C})\leq d_{\hat{T}}(\chi_c,\mathfrak{C})$ where $\hat{T}$ is isomorphic to $P_{n-g,g}$ for some positive integer $g$. This proves the result.
\end{proof}

The following  lemma help us to obtain a path-star tree which maximizes the distance between characteristic center and core center.

\begin{lemma}[\cite{Klp}, Proposition 3.1, 3.2, 3.3 and 3.4]\label{lem:ch2}
Let $n\geq 5$ with $2\leq g\leq n-3$. Then we have the following:
\begin{enumerate}
\item If $\chi_c(P_{n-g,g})= \{i,i+1\}$  then $\chi_c(P_{n-g+1,g-1})=\{i,i+1\}$ or $\{ i+1 \}$ or $\{i+1,i+2\}$;
 
\item If $\chi_c(P_{n-g,g})= \{i\}$  then $\chi_c(P_{n-g+1,g-1})=\{i,i+1\}$;

\item If $\chi_c(P_{n-g,g})= \{i,i+1\}$    then $\chi_c(P_{n-g-1,g+1})=\{i-1,i\}$ or $\{i\}$ or $\{i,i+1\}$;

\item if $\chi_c(P_{n-g,g})= \{i\}$  then $\chi_c(P_{n-g-1,g+1})=\{i-1,i\}$.
\end{enumerate}  
\end{lemma}

\begin{theorem}\label{thm:ch5}
Let T be a tree on $n\geq 6$ vertices. Let $g_0\geq 2$ be the smallest positive integer such that $2^{g_0-1}+g_0\geq n-3$. Then  $d_T(\chi_c,\mathfrak{C})\leq d_{P_{n-g_0,g_0}}(\chi_c,\mathfrak{C})$.
\end{theorem}
\begin{proof}
Let $T$ be a tree on $n\geq 6$ vertices. By Theorem \ref{thm:ch4}, $d_T(C_d,\mathfrak{C})\leq d_{P_{n-g,g}}(C_d,\mathfrak{C})$ for some $g\in \{2,3,\ldots,n-3\}$. Let $g_0\geq 2$ be the smallest positive integer such that $2^{g_0-1}+g_0\geq n-3$  and let $\chi_c(P_{n-g_0,g_0})= \{i\}$ or $\{i,i+1\}$. Then $d_{P_{n-g_0,g_0}}(\chi_c,\mathfrak{C})=i-3$ as $\mathfrak{C}(P_{n-g_0,g_0})=\{2,3\}$ or $\{3\}$.

Let $g'=g_0+1$. By Lemma \ref{lem:cc1}(i), $\mathfrak{C}(P_{n-g',g'})=\{2,3\}$ or $\{3\}$ and hence $d_{P_{n-g',g'}}(\chi_c,\mathfrak{C})=i-3$ or $i-4$ which follows from Lemma \ref{lem:ch2} (3,4). Thus $d_{P_{n-g_0,g_0}}(\chi_c,\mathfrak{C})\geq d_{P_{n-g_0-1,g_0+1}}(\chi_c,\mathfrak{C})$. Continuing in this way, for any $k\geq 1$ with $g_0+k\leq n-3$, we have $d_{P_{n-g_0-k,g_0+k}}(\chi_c, \mathfrak{C})\leq d_{P_{n-g_0,g_0}}(\chi_c, \mathfrak{C})$.

Now let $k'\geq 1$ be an integer such that $g_0-k'\geq 2$.  By Lemma \ref{lem:cc1}, $ p_{g_0-k'}\geq k'+3$.
Then using Lemma \ref{lem:ch2} (1,2), we have $d_{P_{n-g_0+k',g_0-k'}}(\chi_c, \mathfrak{C})\leq i-k'-2\leq i-3 = d_{P_{n-g_0,g_0}}(\chi_c, \mathfrak{C})$. This proves the result
\end{proof}

The exact position of the $\chi_c(P_{n-g,g})$ for every $g$ is not known. So, for tree $T$, getting the strict upper bound for $d_T(\chi_c,\mathfrak{C})$ is difficult. Now we study the asymptotic behaviour of these distances. We define $\delta_n(\chi_c,\mathfrak{C}) = \max\{d_T(\chi_c, \mathfrak{C}) : \mbox{T   is  a tree on $n$ vertices} \}$. Analogously we define $\delta_n(C,\mathfrak{C})$, $\delta_n(C_d,\mathfrak{C})$ and $\delta_n(S_c,\mathfrak{C})$.

\begin{proposition}\label{prop:a1} 
$\displaystyle\lim_{n \to \infty} \frac{\delta_n(C_d,\mathfrak{C})}{n} = \frac{1}{2}.$
\end{proposition}
\begin{proof}
From Theorem \ref{thm:cd2}, we have $\delta_n(C_d,\mathfrak{C})=\lfloor \frac{n-5}{2} \rfloor$. Thus, $$\displaystyle\lim_{n \to \infty} \frac{\delta_n(C_d,\mathfrak{C})}{n} =\displaystyle\lim_{n \to \infty}\frac{\lfloor \frac{n-5}{2} \rfloor}{n}=\frac{1}{2}.$$
\end{proof}
\begin{lemma}\label{lem:ch3}
Let $n\geq 6$. If $g_0$ is the smallest positive integer such that $2^{g_0-1}+g_0\geq n-3$ then $\displaystyle\lim_{n \to \infty}\frac{g_0}{n}=0$.
\end{lemma}
\begin{proof}
Let $g_0$ be the smallest positive integer such that  $2^{g_0-1}+g_0\geq n-3$. Then $2^{g_0-2}+g_0-1< n-3$ and this implies $2^{g_0-2}<n$. Taking logarithm with base $2$ on both side,  we get $g_0-2<\log_2 n$. So, $0<g_0<2+\log_2n$. Therefore $\displaystyle\lim_{n \to \infty}\frac{g_0}{n}=0$ as $\displaystyle\lim_{n \to \infty}\frac{\log_2n}{n}=0$.
\end{proof}
\begin{proposition}\label{prop:a2} 
$\displaystyle\lim_{n \to \infty} \frac{\delta_n(C,\mathfrak{C})}{n} = \frac{1}{2}.$
\end{proposition}
\begin{proof}
From Theorem \ref{thm:cc2}, we have $\delta_n(C,\mathfrak{C})=\lfloor \frac{n-g_0-4}{2} \rfloor$. Now the result follows from Lemma \ref{lem:ch3}
\end{proof}

\begin{lemma}[\cite{Klp}, Theorem 3.3]\label{lem:ch4}
Let $n\geq 6$ with $2\leq g \leq n-3$. Then  $\chi_c(P_{n-g,g})$ lies in the path from $C(P_{n-g,g})$ to $C_d(P_{n-g,g})$.
\end{lemma}

\begin{proposition}\label{prop:a3} 
$\displaystyle\lim_{n \to \infty} \frac{\delta_n(\chi_c,\mathfrak{C})}{n} = \frac{1}{2}.$
\end{proposition}
\begin{proof}
From Theorem \ref{thm:ch5} we have $\delta_n(\chi_c, \mathfrak{C})=d_{P_{n-g_0,g_0}}(\chi_c,\mathfrak{C})$, where $g_0$ is the smallest integer in  satisfying $2^{g_0-1}+g_0\geq n-3$. From Theorem \ref{thm:cc2}, Theorem \ref{thm:cd2} and Lemma \ref{lem:ch4}, we get $\delta_n(\mathfrak{C},C)\leq \delta_n(\mathfrak{C},\chi_c)\leq \delta_n(\mathfrak{C},C_d)$. Thus, $\displaystyle\lim_{n \to \infty}\frac{\delta_n(C,\mathfrak{C})}{n}\leq \displaystyle\lim_{n \to \infty}\frac{\delta_n(\chi_c, \mathfrak{C})}{n}\leq \displaystyle\lim_{n \to \infty}\frac{\delta_n(C_d,\mathfrak{C})}{n}$. Now the result follows form Proposition \ref{prop:a1} and Proposition \ref{prop:a2}.
\end{proof}
\begin{proposition}\label{prop:a4}
$\displaystyle\lim_{n \to \infty} \frac{\delta_n(S_c,\mathfrak{C})}{n} = 1$.
\end{proposition}
\begin{proof}
From Theorem \ref{thm:sc2}, we get $\delta_n(S_c,\mathfrak{C})=n-g_0-3$ where, $n\geq 8$ and $g_0$ be the smallest positive integer such that  $2^{g_0-1}+g_0\geq n-3$. Now the result follows from Lemma \ref{lem:ch3}.
\end{proof}

\noindent{\bf Conflict of interest:} On behalf of all authors, the corresponding author states that there is no conflict of interest.

\end{document}